\documentclass{amsart}
\usepackage{amsmath,amssymb,dsfont}
\usepackage{graphicx,color}
\usepackage{pstricks, pst-plot, pst-node, pst-grad, pst-math}
\usepackage{pst-plot}
\usepackage{todonotes}
\usepackage{comment}
\newtheorem{theorem}{Theorem}[section]
\newtheorem{lemma}[theorem]{Lemma}
\newtheorem{proposition}[theorem]{Proposition}
\newtheorem{corollary}[theorem]{Corollary}

\theoremstyle{definition}
\newtheorem{definition}[theorem]{Definition}

\newtheorem{remark}[theorem]{Remark}

\numberwithin{equation}{section}

\newcommand{\rd}{{\,\rm d}}
\newcommand{\e}{{\rm e}}

\newcommand{\N}{{\mathbb N}}

\newcommand{\R}{{\mathbb R}}
\newcommand{\C}{{\mathbb C}}

\newcommand\beq{\begin{equation}}
\newcommand\eeq{\end{equation}}

\newcommand{\dist}{\mathrm{dist}}
\newcommand\re{\mathrm{Re}}

\newcommand\I{\mathrm{i}}
\newcommand{\beqnt}{\begin{equation*}}
\newcommand{\eeqnt}{\end{equation*}}
\newcommand{\set}[2]{\{#1 : #2 \}}

\DeclareMathOperator{\supp}{supp}

\DeclareMathOperator{\diam}{diam}

\begin{document}
\author{Jean-Claude Cuenin}
\address{Mathematisches Institut, Ludwig-Maximilians-Universit\"at M\"unchen, 80333 Munich, Germany}
\email{cuenin@math.lmu.de}

\title{Eigenvalue estimates for bilayer graphene}

\begin{abstract}
Recently, Ferrulli-Laptev-Safronov \cite{2016arXiv161205304F} obtained eigenvalue estimates for an operator associated to bilayer graphene in terms of $L^q$ norms of the (possibly non-selfadjoint) potential. They proved that for $1<q<4/3$ all non-embedded eigenvalues lie near the edges of the spectrum of the free operator. In this note we prove this for the larger range $1\leq q\leq 3/2$. The latter is optimal if embedded eigenvalues are also considered. We prove similar estimates for a modified bilayer operator with so-called ``trigonal warping" term. Here, the range for $q$ is smaller since the Fermi surface has less curvature. The main tool are new uniform resolvent estimates that may be of independent interest and are collected in an appendix (in greater generality than needed).   
\end{abstract}

\maketitle

\section{Introduction and main results}

Graphene is a two-dimensional material consisting of a single layer of carbon atoms arranged in a honeycomb lattice. Behind its many remarkable properties is the fact that at low-energy the dynamics of charge carriers in graphene is governed by the Dirac equation. \emph{Bilayer graphene} is produced by stacking one layer on top of another, with a $60$ degrees angle of rotation between the two layers. The effective Hamiltonian of bilayer graphene, in suitable units, is given by 
$$
D_m=
\begin{pmatrix}
m&4\partial_{\bar z}^2\\4\partial_{ z}^2&-m
\end{pmatrix}, 
$$
where 
$$
 \partial_{\bar z}=\frac12\Bigl(\frac{\partial}{\partial x_1}+\frac1i\frac{\partial}{\partial x_2}\Big),  \qquad \partial_{ z}=\frac12\Bigl(\frac{\partial}{\partial x_1}-\frac1i\frac{\partial}{\partial x_2}\Big),
$$
see e.g.\ formula (1.46) in \cite{katsnelson_2012} and the original references mentioned there.
The parameter $m\geq 0$ plays the role of a mass term, similar as for the Dirac operator. However, it should not be confused with the physical mass (denoted by $m^*$ in \cite{katsnelson_2012}) of the charge carrier (i.e.\ the electron). Experimentally, a mass gap can be created by physical strain, see Vozmediano et al~\cite{Vozmedianoetal2010}.

We will consider the operator $D:=D_m+V$ on the Hilbert space $L^2(\R^2; \C^2)$. The domain of $D$ is  the  Sobolev  space $H^2(\R^2; \C^2)$.
The potential $V:\R^2\to \mathrm{Mat}(2\times 2;\C)$ is not assumed to take values in the self-adjoint matrices, i.e\ $D$ is allowed to be non-selfadjoint. From the point of view of physics an important motivation to consider non-self-adjoint operators comes from the study of resonances, either via the complex scaling method introduced by Aguilar-Combes \cite{MR0345551} or via the method of complex absorbing potentials (see e.g\ Riss-Meyer \cite{MR1252340} or Zworski \cite{2015arXiv150500721Z}.)

Since $D_m^2=\Delta^2+m^2$, the spectrum of $D_m$ is absolutely continuous and given by
$$\sigma(D_m)= (-\infty, m]\cup [m,\infty).$$
Note in particular that $\sigma(D_0)=\R$.

Laptev-Ferrulli-Safronov \cite{2016arXiv161205304F} proved the following eigenvalue estimate: If $V\in L^{q}(\R^2,\mathrm{Mat}(2\times2;\C))$ with $1<q<4/3$ and if $z\in\sigma_{\rm p}(D)\setminus\sigma(D_m)$, then 
\begin{align}\label{eq. LFS}
C_q\|V\|_q^q|z^2-m^2|^{(1-q)/2}\Bigl(\sqrt{\Bigl|\frac{z-m}{z+m}\Bigr|}+\sqrt{\Bigl|\frac{z+m}{z-m}\Bigr|}+1\Bigr)^q\geq 1.
\end{align}
Here, $C_q>0$ is independent of $V$, $z$ and $m$. Our first result is is that \eqref{eq. LFS} is valid for the larger range $1\leq q\leq 3/2$ and for any $z\in\sigma_{\rm p}(D)$. 
In view of the counterexamples in \cite{arXiv:1709.06989} this range is optimal if embedded eigenvalues are taken into account. Whether it is optimal for $z\in\sigma_{\rm p}(D)\setminus\sigma(D_m)$ is related to the Laptev-Safronov conjecture for Schr\"odinger operators \cite{MR2540070}. We refer to \cite{MR2820160,MR3713021,MR3717979} for a full discussion and progress towards a resolution of the conjecture.

We adopt the following notation, in line with \cite{MR3177918}, \cite{Cuenin2018}.
\begin{align}
k(z)^4&:=z^2-m^2,\quad\arg\left(k(z)\right)\in [0,\pi/2),\label{def. k}\\
\zeta(z)&:=\frac{z+m}{k(z)^2},\quad z\neq\pm m.\label{def. zeta}
\end{align}
Observe that $\zeta^2$ extends to a holomorphic map on the Riemann sphere $\C\cup\{\infty\}$, with $\zeta^2(m)=\infty$ and $\zeta^2(-m)=0$. We abbreviate
\begin{align*}
\|V\|_q^q:=\int_{\R^2}\|V(x)\|^q\rd x,\quad 1\leq q<\infty,
\end{align*}
Here, $\|V(x)\|$ is the operator norm of $V(x)\in \mathrm{Mat}(2\times2;\C)$ as a linear map on $\C^2$, when the latter is equipped with the Euclidean norm. 

\begin{theorem}\label{thm. 1}
Let $V\in L^{q}(\R^2,\mathrm{Mat}(2\times2;\C))$, with $1\leq q\leq 3/2$. Then the following estimates hold for $D:=D_m+V$ with some constant $C_q>0$ independent of $V$, $z$ and $m$.\\

\noindent
{\rm(i)} If $1< q\leq 3/2$, then every $z\in\sigma_{\rm p}(D)$ satisfies
\begin{align}\label{eq. bound thm. 1 i)}
\frac{|k(z)|^{2q-2}}{\left(1+|\zeta(z)|+|\zeta(z)|^{-1}\right)^{q}}\leq C_{q}\|V\|_{q}^{q}.
\end{align}
In particular, if $m=0$, the left hand side is proportional to $|z|^{q-1}$.\\

\noindent
{\rm(ii)} If $q=1$, then there exists $c_1>0$ such that if $\|V\|_1\leq c_1$, then 
\begin{align}\label{eq. bound thm. 1 ii)}
\sigma_{\rm p}(D)\subset \set{z\in\C}{|z\pm m|\leq C_1 m\|V\|_1^2}\quad \mbox{if  }m>0,
\end{align}
whereas for $m=0$ we have $\sigma_{\rm p}(D)=\emptyset$.\footnote{In fact, if $m=0$ and $\|V\|_1$ is sufficiently small, then $D$ and $D_0$ are similar. This follows from Katos theory of smooth perturbations \cite[Theorem 1.5]{MR0190801} and the fact that \eqref{eq. resolvent Schatten} is uniformly bounded for $m=0$ and $q=1$.}
\end{theorem}

\begin{remark}
The range $1<q\leq 3/2$ in Theorem \ref{thm. 1} is the same as for the Schr\"odinger operator $-\Delta+V$ in two dimensions, with $V\in L^q(\R^2)$, for which the inequality
\begin{align}\label{Frank d=2}
|z|^{q-1}\leq C_q\|V\|_q^q
\end{align}
was proved by Frank \cite{MR3556444}. This is no coincidence since $D_m$ and $-\Delta$ are both second order differential operators. Note, however, that \eqref{Frank d=2} is \emph{false} for $q=1$, even for real-valued potentials. Indeed, it is well known that for $V\leq 0$ and not identically zero, $-\Delta+V$ always has a negative eigenvalue. The failure of the $q=1$ bound stems from the logarithmic singularity of the fundamental solution of $-\Delta$. In contrast, the fundamental solution of $D_m$ is \emph{bounded}, see \eqref{eq. cancellation}. A more intuitive way to see this is to observe that the fundamental solution of $\partial_z^2$, given by $z/\overline{z}$, is bounded.  
\end{remark}

Our second result deals with a modified version of the operator $D_m$; we consider the case $m=0$ for simplicity, but the case $m>0$ could be treated along the same lines. Recall that $D_0$ is an effective low-energy Hamiltonian for bilayer graphene. It arises as the leading order approximation of a tight-binding Hamiltonian, see e.g\ \cite{katsnelson_2012}.
If one takes higher order terms (larger-distance hopping processes) into account the effective Hamiltonian is modified to
\begin{align*}
D_{\rm trig}=
\begin{pmatrix}
0&4\partial_{\bar z}^2+2\alpha \partial_{ z}\\4\partial_{ z}^2+2\alpha \partial_{\bar z}&0
\end{pmatrix},
\end{align*}
where $\alpha>0$ depends on the physical paramters of the model, see (1.53) in \cite{katsnelson_2012}. We will set $\alpha=1$ for simplicity; its original value can easily be restored by dimensional analysis. The eigenvalues $\lambda_{\pm}(\xi)$ of the symbol of $D_{\rm trig}$ are given by
\begin{align}\label{def. P}
\lambda_{\pm}(\xi):=\pm \sqrt{P(\xi)},\quad \mbox{where}\quad P(\xi):=|\xi|^4+2\re(\xi^3)+|\xi|^2.
\end{align}
The cubic term can be written in polar coordinates $\xi=|\xi|\e^{\I\theta}$ as $2|\xi|^3\cos(3\theta)$. Therefore, the dispersion relations $\lambda_{\pm}(\cdot)$ are only invariant under $120$ degree rotations. This effect is called ``trigonal warping" in the physical literature and explains the subscript ``trig". 

\begin{theorem}\label{thm. 2}
Let $V\in L^{q}(\R^2,\mathrm{Mat}(2\times2;\C))$, with $1\leq q\leq 3/2$. Then the following estimates hold for $D:=D_{\rm trig}+V$ with some constant $C_q>0$ independent of $V$ and $z$.\\ 

\noindent
{\rm(i)} We have the inclusion
\begin{align}\label{spectral inclusion trig. warp. 1}
\sigma_{\rm p}(D)\subset\set{z\in\C}{|z|^{q-1}\leq C_q(1+\|V\|_q^q)}.
\end{align}
{\rm(ii)} There exists $c>0$ such that if $V\in (L^{q}\cap L^{q_0})(\R^2,\mathrm{Mat}(2\times2;\C))$, where $1< q\leq 3/2$ and $q_0=\min(q,5/4)$, and if $\|V\|_{q_0}\leq c$, then
\begin{align}\label{spectral inclusion trig. warp. 2}
\sigma_{\rm p}(D)\subset\set{z\in\C}{|z|^{q-1}\leq C_q\|V\|_q^q}\cup \set{z\in\C}{|z-1/16|^{q-1}\leq C_q\|V\|_q^q}.
\end{align}
{\rm(iii)} If $q=1$, then there exists $c_1>0$ such that if $0<\|V\|_{1}\leq c_1$, then
\begin{align}\label{spectral inclusion trig. warp. 3}
\sigma_{\rm p}(D)\subset\set{z\in\C}{|z-1/16|\leq \exp(-C_1\|V\|_1^{-1}
)}.
\end{align}
\end{theorem}

\begin{remark}
In view of the counterexamples in \cite{arXiv:1709.06989}, (i) is optimal in the sense that \eqref{spectral inclusion trig. warp. 1} cannot hold for any $q>3/2$. The situation regarding (ii), (iii) is less clear. Remark \eqref{remark optimality of k=4 for trig} shows that there exist $\lambda>0$ for which the Fermi surface $M_{\lambda}=\set{\xi\in\R^2}{P(\xi)=\lambda}$ has points of vanishing curvature. Theorem 1.2 in \cite{arXiv:1709.06989} then shows that \eqref{spectral inclusion trig. warp. 2} cannot hold for $q>4/3$ if the assumption $V\in L^{5/4}(\R^2,\mathrm{Mat}(2\times2;\C))$ is dropped. An open problem is thus whether \eqref{spectral inclusion trig. warp. 2} is true under the assumption that $V\in L^{q}(\R^2,\mathrm{Mat}(2\times2;\C))$, with $5/4< q\leq 4/3$. Concerning (iii), another open problem is whether the stronger conclusion $\sigma_{\rm p}(D)=\emptyset$ holds provided $c_1$ is sufficiently small. The obstruction in our proof is the logarithmic loss in the constant of Proposition \ref{prop. schatten estimate low frequency critical point}. 
\end{remark}

The outline of this note is as follows. In section 2 we prove uniform resolvent estimates in Schatten spaces for the bilayer operators $D_m$ and $D_{\rm trig}$. These follow from the more general results in the appendix. In Section 3 we prove Theorems \ref{thm. 1}--\ref{thm. 2}.

Even though operator norm estimates would be sufficient for our purpose here, we state our resolvent estimates with the stronger Schatten norm bounds for future reference. Such bounds were pioneered by Frank-Sabin \cite{MR3730931} for the Laplacian.

\section{Resolvent estimates}

\subsection{Fourier transforms of arclength measures}

\begin{definition}\label{def. finite type k 2d}
A smooth curve $C\subset \R^2$ is called of finite type at $\xi_0\in C$ if its curvature $\kappa$ does not vanish to infinite order at $\xi_0$. The smallest $k\in\N_{\geq 2}$ such that $\kappa^{(k-2)}(\xi_0)\neq 0$ is called the type of $C$ at $\xi_0$. For any compact subset $C'\subset C$ The type of $C'$ is defined as the maximum of the types of $\xi_0\in C'$. 
\end{definition}

\begin{remark}\label{rem. h}
Suppose that, after translation and rotation, $\xi_0=0$ and $C$ is locally the graph of a smooth map $h:U\to\R^2$ where $U$ is a neighborhood of the origin in~$\R$ and $h(0)=h'(0)=0$. Then 
\begin{align}\label{def. kappa}
\kappa=\frac{h''}{(1+(h')^2)^{3/2}}.
\end{align}
Hence the type of $C$ at the origin equals the smallest $k\in\N_{\geq 2}$ such that $h^{(k)}(0)\neq 0$.
\end{remark}

\begin{lemma}\label{lemma decay of the FT}
Let $C\subset \R^2$ be a smooth compact curve of type $k$. Then the Fourier transform of the arclength measure on $C$ is $\mathcal{O}((1+|x|)^{-1/k})$.
\end{lemma}

\begin{proof}
This follows from van der Corput's lemma, see e.g.\ Proposition 2 (and its corollary) in Section VIII of \cite{MR1232192}.
\end{proof}

\subsection{Resolvent estimates for $D_m$}

In the following, $\alpha_{q,r,d}\in (1,\infty)$ is as defined in \eqref{def. alpha_qrd}.

\begin{proposition}\label{prop. resolvent Schatten no trig. warp.}
Let $1\leq q\leq 3/2$. There exists $C>0$ such that for any $A,B\in L^{2q}(\R^2,\mathrm{Mat}(2\times2;\C))$ and $z\in\rho(D_m)$ we have the inequality
\begin{align}\label{eq. resolvent Schatten}
\|A(D_m-z)^{-1}B\|_{\mathfrak{S}^{\max(q,\alpha_{q,1/2,2})}}\leq C|k(z)|^{\frac{2}{q}-2}(|\zeta(z)|+|\zeta(z)|^{-1})\|A\|_{2q}\|B\|_{2q}.
\end{align}
\end{proposition}

\begin{proof}
By scaling, we may assume that $|k(z)|=1$. By a limiting argument, we may also assume that
$k(z)\neq 1$, so $k(z)=\e^{\I\varphi}$ with $\varphi\in (0,\pi/2)$. Since $D_m^2=(\Delta^2+m^2)\mathbf{1}_{2\times 2}$, we have the identity
\begin{align}\label{resolvent identity}
(D_m-z)^{-1}=(D_m+z)(\Delta^2-k(z)^4)^{-1}.
\end{align}
It is therefore enough to prove the estimate \eqref{eq. resolvent Schatten} with $(D_m-z)^{-1}$ replaced by the Fourier multipliers $m_j(D)$, $j=1,2$, with 
\begin{align*}
m_1(\xi):=\frac{1}{|\xi|^4-\e^{\I\varphi}},\quad
m_2(\xi):=\frac{(\xi_1\pm\I\xi_2)^2}{|\xi|^4-\e^{\I\varphi}}.
\end{align*}
Let $\chi_1\in C_c^{\infty}(\R^2;[0,1])$ be supported on $\set{\xi\in\R^2}{1/2\leq |\xi|\leq 3/2}$ and equal to $1$ on $\set{\xi\in\R^2}{3/4\leq |\xi|\leq 5/4}$, and set $\chi_2:=1-\chi_1$. Then it remains to prove the estimates
\begin{align}\label{eq. resolvent Schatten i,j}
\|A\chi_i(D)m_j(D)B\|_{\mathfrak{S}^{\alpha_q}}\leq C \|A\|_{2q}\|B\|_{2q},\quad i,j=1,2,
\end{align}
with $C$ independent of $\varphi$. Since the zero set of $|\xi|^4-1$ is the unit circle the case $i=1$ follows from Corollary~\ref{cor. schatten estimate low frequency principal curvature} or \ref{cor. schatten estimate low frequency finite type}.\footnote{In fact, since $2(|\xi|^4-k^4)^{-1}=(|\xi|^2-k^2)^{-1}-(|\xi|^2+k^2)^{-1}$ it would also follow from Kenig-Ruiz-Sogge \cite{MR894584}.} The case $i=2$ and $1<q\leq 3/2$ follows from the Kato-Seiler-Simon inequality \cite[Theorem 4.1]{MR2154153} and H\"older's inequality in Schatten spaces. In fact, for $q>1$, we have
\begin{align*}
\|A\chi_2(D)m_j(D)B\|_{\mathfrak{S}^{q}}\leq C\|(|\cdot|^2+1)^{-1}\|_q\|A\|_{2q}\|B\|_{2q},
\end{align*}
which is again better than the claimed bound. The case $j=2$ and $q=1$ would follow from the estimate
\begin{align}\label{eq. cancellation}
\sup_{x\in\R^2\setminus\{0\}}\left|\int_{\R^2}\e^{\I x\cdot\xi}\frac{(\xi_1\pm\I\xi_2)^2}{|\xi|^4+1}\rd\xi\right|<\infty
\end{align}
since the latter implies boundedness of the kernel of $\chi_2(D)m_2(D)$ and hence the Hilbert-Schmidt bound
\begin{align*}
\|A\chi_2(D)m_2(D)B\|_{\mathfrak{S}^{2}}\leq C \|A\|_{2q}\|B\|_{2q}.
\end{align*}
To prove \eqref{eq. cancellation} fix $x\in\R^2\setminus\{0\}$ and choose polar coordinates $\xi=r\e^{\I\theta}$ where $\theta$ is the angle between $\xi$ and $x$. The integral in \eqref{eq. cancellation} then becomes
\begin{align*}
\int_{0}^{\infty}\left(\int_0^{2\pi}\e^{\I r|x|\cos(\theta)}\e^{2\I\theta}\rd\theta\right)\frac{r^3}{r^4+1}\rd r.
\end{align*}
Integration by parts shows that the $\theta$-integral is $\mathcal{O}(r|x|)$ for $r|x|\leq 1$. Stationary phase estimates yield an $\mathcal{O}((r|x|)^{-1/2})$ bound For $r|x|>1$. Splitting the $r$-integral into a part where $r\leq |x|^{-1}$ and a part where $r>|x|^{-1}$ then yields \eqref{eq. cancellation}.
\end{proof}

\subsection{Resolvent estimates for $D_{\rm trig}$}

In the following we consider the symbol $P$, given by \eqref{def. P}, as well as the curve (the ``Fermi surface")
$$M_{\lambda}=\set{\xi\in\R^2}{P(\xi)=\lambda}.$$

\begin{lemma}\label{lemma critical points bilayer}
The set of critical values of $P$ is $\{0,1/16\}$. The set of critical points is $\mathcal{C}=\mathcal{C}_{0}\cup \mathcal{C}_{1}$, where
\begin{equation}\label{def. critical points}
\begin{split}
\mathcal{C}_{0}&=P^{-1}(\{0\})\cap\mathcal{C}=\{0,e^{\I\pi/3},e^{\I\pi},e^{\I 5\pi/3}\},\\
\mathcal{C}_{1}&=P^{-1}(\{1/16\})\cap\mathcal{C}=\{e^{\I\pi/3}/2,e^{\I\pi}/2,e^{\I 5\pi/3}/2\}.
\end{split}
\end{equation} 
All critical points are non-degenerate. In fact, the points in $\mathcal{C}_0$ are minima, and the points in $\mathcal{C}_1$ are saddle points. $M_{\lambda}$ is compact for all $\lambda\in\R_+\setminus \{0,1/16\}$, has four connected components for $0<\lambda<1/16$ and is connected for $\lambda>1/16$. 
\end{lemma}

\begin{proof}
In polar coordinates $\xi=r\e^{\I\theta}$ we have 
\begin{align*}
\widetilde{P}(r,\theta):=P(r\e^{\I\theta})=r^4+2r^3\cos(3\theta)+r^2
\end{align*}
The $(r,\theta)$-gradient is
\begin{align*}
\nabla_{(r,\theta)} \widetilde{P}(r,\theta)
=\left(4r^3+6r^2\cos(3\theta)+2r\right)e_r+\left(-6r^2\sin(3\theta)\right)e_{\theta}.
\end{align*}
It vanishes whenever
\begin{align}\label{critical points}
r=0,\quad \mbox{or}\quad (\sin(3\theta)=0\mbox{  and  } 4r^3-6r^2+2r=0),
\end{align}
where we used that $\cos(3\theta)=\pm 1$ whenever $\sin(3\theta)=0$, and that the equation $\nabla_{(r,\theta)} \widetilde{P}(r,\theta)=0$ with $\cos(3\theta)=1$ does not have a solution $r>0$. The angles where $\cos(3\theta)=-1$ are $\theta_j=\pi(1+2j)/3$, $j=0,1,2$.
The nonzero solutions of the rightmost equation in \eqref{critical points} are $r=1$ and $r=1/2$. 
One easily checks that the critical values are $0$ and $1/16$ and that the Hessian is nondegenerate at the critical points. 
Compactness of $M_{\lambda}$ follows from the fact that $P(\xi)\to+\infty$ as $|\xi|\to+\infty$. The claim about the number of connected components follows from Morse theory.  
\end{proof}

\begin{lemma}\label{lemma k=4}
For any $\lambda\in\R\setminus\{0,1/16\}$ the curve $M_{\lambda}$ is of finite type $k=4$.
\end{lemma}

\begin{proof}
After rotation and translation we may assume that $M_{\lambda}$ is locally a graph of a smooth function $h:U\to\R$ over the $\xi_2$ axis such that $h(0)=h'(0)=0$. By Remark \ref{rem. h} it suffices to show that
\begin{align}\label{Pol. of deg. 4}
\sum_{k=2}^4|h^{(k)}(0)|\neq 0.
\end{align}
Assume the right hand side of \eqref{Pol. of deg. 4} is zero. Repeatedly differentiating the identity $P(h(\xi_2),\xi_2)=\lambda$ at the point $\xi_2=0$ yields
\begin{align}\label{derivatives of h vanish}
0=h^{(k)}(0)=-\frac{\partial_{\xi_2}^{(k)}P(0)}{\partial_{\xi_1}P(0)},\quad k=1,\ldots,4.
\end{align}
Since $P$ is a polynomial of degree four this implies that $t\mapsto P(te_2)$ is constant. However, for any rotation matrix $R$ and for any $a\in\R^2$, the $t^4$-coefficient of the polynomial $t\mapsto P(tRe_2+a)$ is equal to $1$. Undoing the translation and rotation, we arrive at a contradiction.
\end{proof}

\begin{remark}\label{remark optimality of k=4 for trig}
It may be seen by explicit computation that $k=4$ is optimal in general. Indeed, changing variables to $\xi_1=\frac{1}{2}(3-\sqrt{7})+s$, $\xi_2=t$, one finds that
\begin{align*}
P(\frac{1}{2}(3-\sqrt{7})+s,t)=\frac{233}{4}-22\sqrt{7}+t^4+(72-27\sqrt{7}+\mathcal{O}(t^2))s+\mathcal{O}(s^2).
\end{align*} 
Then, for $\lambda:=\frac{233}{4}-22\sqrt{7}$ and $n=1,2,3$, we have
\begin{align*}
P(\frac{1}{2}(3-\sqrt{7}),0)=\lambda,\quad \partial_s P(\frac{1}{2}(3-\sqrt{7}),0)\neq 0,\quad \partial_t^nP(\frac{1}{2}(3-\sqrt{7}),0)=0,
\end{align*}
and thus, by a parallel argument as that leading to \eqref{derivatives of h vanish}, the curvature of $M_{\lambda}$ vanishes to second order at $\xi=(\frac{1}{2}(3-\sqrt{7}),0)$.   
\end{remark}

We remind the reader that $\alpha_{q,r,d}\in (1,\infty)$ is defined in \eqref{def. alpha_qrd}.

\begin{proposition}\label{prop. resolvent Schatten trig. warp. critical points}
Let $1\leq q\leq 3/2$ and let $\mathcal{C}$ be given by \eqref{critical points}. There exist $\delta_1>0$ and $C_q^{(1)}>0$ such that for fixed $\chi_1\in C_c^{\infty}(B_{\delta_1}(\mathcal{C}))$ the following holds: For any $A,B\in L^{2q}(\R^2,\mathrm{Mat}(2\times2;\C))$ and $z\in\rho(D_{\rm trig})$ we have the inequality
\begin{align}\label{eq. resolvent Schatten trig. warp. critical points}
\|A\chi_1(D)(D_{\rm trig}-z)^{-1}B\|_{\mathfrak{S}^{\alpha_{q,1/2,2}}}\leq C_q^{(1)}N_q(z)\|A\|_{2q}\|B\|_{2q},
\end{align}
where $N_q:\C\setminus\{0,1/16\}\to\R_+$ is a continuous function satisfying
\begin{align}\label{def Nq}
N_q(z)=\begin{cases}
|z|^{\frac{1}{q}-1}&\quad\mbox{if  }|z|\leq 1/64,\\
\upsilon(z)|z-1/16|^{\frac{1}{q}-1}&\quad\mbox{if  }|z-1/16|\leq 1/64.
\end{cases}
\end{align} 
with $\upsilon(z)=1$ or $\upsilon(z)=-\ln|z-1/16|$ according to whether $q>1$ or $q=1$.
\end{proposition}

\begin{proof}
We cut up the support of $\chi_1$ by a finite partition of unity $1=\sum_j\psi_j$ such that the support of each $\psi_j$ contains exactly one critical point; this is possible by choosing $\delta_1$ sufficiently small. The claim thus follows from Lemma \ref{lemma critical points bilayer} and Proposition \ref{prop. schatten estimate low frequency critical point}, together with the identity \eqref{resolvent identity trig. warp.}.
\end{proof}

\begin{proposition}\label{prop. resolvent Schatten trig. warp. infinity}
Let $1\leq q\leq 3/2$ and $\chi_2\in C^{\infty}(\R^d)$. There exist $\delta_2>0$ and $C_q^{(2)}>0$ such that for any $A,B\in L^{2q}(\R^2,\mathrm{Mat}(2\times2;\C))$ we have the inequality
\begin{align}\label{eq. resolvent Schatten trig. warp. infinity}
\|A\chi_2(D)(D_{\rm trig}-z)^{-1}B\|_{\mathfrak{S}^{\max(q,\alpha_{q,1/2,2})}}\leq C_q^{(2)}(1+|z|)^{\frac{1}{q}-1}\|A\|_{2q}\|B\|_{2q}
\end{align}
uniformly in $z\in\rho(D_{\rm trig})$ provided 
\begin{align*}
|z|+\inf\set{|\xi|}{\xi\in\supp\chi_2}\geq 1/\delta_2.
\end{align*}
\end{proposition}

\begin{proof}
The proof is similar to that of \eqref{eq. resolvent Schatten}. Instead of homogeneity, one uses a rescaling argument such as that in the proof of Proposition~\ref{prop. schatten estimate low frequency critical point}. We omit the details.
\end{proof}

\begin{proposition}\label{prop. resolvent Schatten trig. warp. away from crit. points}
Let $1\leq q\leq 5/4$, $\chi_3\in C_c^{\infty}(\R^d\setminus\mathcal{C})$, where $\mathcal{C}$ is given by \eqref{critical points}.
Then there exists $C_q^{(3)}>0$ such that for any $A,B\in L^{2q}(\R^2,\mathrm{Mat}(2\times2;\C))$ and $z\in\rho(D_{\rm trig})$ we have the inequality
\begin{align}\label{eq. resolvent Schatten trig. warp. away from crit. points}
\|A\chi_3(D)(D_{\rm trig}-z)^{-1}B\|_{\mathfrak{S}^{\alpha_{q,1/4,2}}}\leq C_q^{(3)}(1+|z|)^{-1}\|A\|_{2q}\|B\|_{2q}.
\end{align}
\end{proposition}

\begin{proof}
Since $D_{\rm trig}^2=P(D)\mathbf{1}_{2\times 2}$, with $P$ given by \eqref{def. P}, we have the identity
\begin{align}\label{resolvent identity trig. warp.}
(D_{\rm trig}-z)^{-1}=(D_{\rm trig}+z)(P(D)-z^2)^{-1}.
\end{align}
In view of the localization $(D_{\rm trig}+z)$ contributes only a factor $\mathcal{O}((1+|z|))$ to the estimate. It is thus sufficient to prove that
\begin{align*}
\|Am(D)B\|_{\mathfrak{S}^{\alpha_q}}\leq C\|A\|_{2q}\|B\|_{2q},
\end{align*}
where
\begin{align*}
m(\xi):=\frac{(1+|z|^2)\chi_3(\xi)}{P(\xi)-z^2}
\end{align*}
Define
\begin{align*}
a:=\max_{\xi\in\supp\chi_3}P(\xi)<+\infty.
\end{align*}
The estimate is trivial for $|z|^2>2a$ since $m(D)$ is then a smoothing operator with uniform bounds in $z$. Assume now that $|z|^2\leq 2a$, and let $\psi\in C_c^{\infty}(\R^2;[0,1])$ be such that 
\begin{align*}
\psi\equiv 1\quad\mbox{on}\quad\bigcup_{\lambda^2\leq 4a}M_{\lambda}
\end{align*}
Then $(1-\psi(D))m(D)$ is a smoothing operator with uniform bounds in $|z|\leq 2a$, and the estimate is again trivial for this part. The estimate for the remaining part $\psi(D)m(D)$ follows from Lemma~\ref{lemma k=4} and Corollary~\ref{cor. schatten estimate low frequency finite type} with $k=4$. 
\end{proof}

\section{Proofs of Theorems \ref{thm. 1}--\ref{thm. 2}}

\begin{proof}[Proof of Theorem \ref{thm. 1}]
Let $A:=|V|^{1/2}$, $B:=U|V|^{1/2}$ where $U$ is the partial isometry in the polar decomposition of $V$. By Proposition \ref{prop. resolvent Schatten no trig. warp.} and the Birman-Schwinger principle it follows that for $z\in\sigma_{\rm p}(D)\setminus\sigma(D_m)$ we have
\begin{align*}
1\leq \|A(D_m-z)^{-1}B\|\leq C|k(z)|^{\frac{2}{q}-2}(|\zeta(z)|+|\zeta(z)|^{-1})\|V\|_{q}.
\end{align*} 
This immediately yields \eqref{eq. bound thm. 1 i)}.\footnote{Strictly speaking, we have only proved \eqref{eq. bound thm. 1 i)} for $z\notin\sigma(D_0)$. However, embedded eigenvalues can be included by \cite[Proposition 3.1]{MR3713021}. This applies to the proof of Theorem \ref{thm. 2} as well.} The case $z\in\sigma_{\rm p}(D)$ follows from \cite[Proposition 3.1]{MR3713021}.
One obtains \eqref{eq. bound thm. 1 ii)} from \eqref{eq. bound thm. 1 i)} by solving for $|\zeta(z)|$ and Taylor expanding in $\|V\|_1$. This kind of computation may be found in the proof of Theorem 2.1 in~\cite{MR3177918} and is therefore omitted. The claim $m=0$ is obvious since $|\zeta(z)|=1$ in this case.
\end{proof}

\begin{proof}[Proof of Theorem \ref{thm. 2}]
Let $\mathcal{C}$ be the set of critical points of $P$, see \eqref{critical points}. Let $C_q^{(1)},C_q^{(2)},C_q^{(3)}$, $\delta_1,\delta_2$ be as in Propositions \ref{prop. resolvent Schatten trig. warp. critical points}--\ref{prop. resolvent Schatten trig. warp. away from crit. points} and set $\delta:=\min(\delta_1,\delta_2)$. Let $z\in \sigma_{\rm p}(D)\cap \rho(D_{\rm trig})$. Proposition \ref{prop. resolvent Schatten trig. warp. infinity} with $\chi_2\equiv 1$ and the Birman-Schwinger principle imply that either $|z|\leq \delta^{-1}$ or that $1\leq C_q^{(2)} |z|^{\frac{1}{q}-1}\|V\|_q$.
Hence,
\begin{align*}
|z|^{q-1}\leq \max(\delta^{1-q},(C_q^{(2)})^q\|V\|_q^q),
\end{align*}
proving \eqref{spectral inclusion trig. warp. 1}. To prove~\eqref{spectral inclusion trig. warp. 2}, we may assume that 
\begin{align}\label{Proof trig warp. 0}
\dist(z,\{0,1/16\})\leq \min(1/64,(C_q^{(1)}/C_q^{(2)})^{q/(q-1)}),
\end{align}
otherwise the claim follows from \eqref{spectral inclusion trig. warp. 1}.
pick a partition of unity $1=\chi_1+\chi_2+\chi_3$ subordinate to the decomposition
\begin{align*}
\R^2=B_{\delta}(\mathcal{C})\cup (\R^2\setminus B_{1/\delta}(0))\cup (B_{1/\delta}(0)\setminus B_{\delta}(\mathcal{C})).
\end{align*}
Birman-Schwinger together with Propositions \ref{prop. resolvent Schatten trig. warp. critical points}--\ref{prop. resolvent Schatten trig. warp. away from crit. points} yields
\begin{align}\label{Proof trig warp. 1}
1\leq \sum_{j=1}^3\|A(D_{\rm trig}-z)^{-1}B\|\leq (C_q^{(1)}N_q(z)+C_q^{(2)})\|V\|_q+C_q^{(3)}c_1,
\end{align}
where $N_q(z)$ is defined in \eqref{def Nq}. By \eqref{Proof trig warp. 0} we have $N_q(z)\geq C_q^{(2)}/C_q^{(1)}$, hence, if $c_1<(C_q^{(3)})^{-1}$, then
\begin{align}\label{Proof trig warp. 2}
N_q(z)^{-q}\leq \left(\frac{2C_q^{(2)}}{1-C_q^{(3)}c_1}\right)^q\|V\|_q^q,
\end{align}
which implies \eqref{spectral inclusion trig. warp. 2}.
If $q=1$, one has to multiply $N_q(z)$ by $-\ln|z-1/16|$ in \eqref{Proof trig warp. 1} to obtain \eqref{spectral inclusion trig. warp. 3} analogously.
\end{proof}

\appendix

\section{Frequency-localized resolvent estimates in $d\geq 2$}

\begin{definition}
Let $S\subset\R^d$ be a smooth hypersurface. A \emph{defining function} of $S$ is a smooth function $\rho:\R^d\to\R$ such that $S=\set{\xi\in\R^d}{\rho(\xi)=0}$ and $|\nabla\rho|> 0$ on $S$. If in addition $|\nabla\rho|=1$ on $S$, we call $\rho$ a \emph{normalized defining function} of $S$. 
\end{definition}

\begin{definition}
Let $S\subset\R^d$ be a smooth hypersurface with normalized defining function $\rho:\R^d\to\R$. The second fundamental form of $S$ at $\xi_0\in S$ is the restriction of $\nabla^2\rho(\xi_0)$ to the tangent space $T_{\xi_0}S\subset\R^d$, and its eigenvalues are the principal curvatures at $\xi_0$. The Gaussian curvature is the product of all principal curvatures.  
\end{definition}

\begin{definition}
A smooth hypersurface $S\subset\R^d$ is called of finite type at $\xi_0\in S$ if at least one of its principal curvatures does not vanish to infinite order at $\xi_0$.
\end{definition}

\begin{remark}
In analogy to Remark \ref{rem. h}, if $S$ is locally the graph of a smooth function $h:U\to \R$ defined near the origin and with $h(0)=0$, $\nabla h(0)=0$, then the type $k$ of $S$ at the origin equals the smallest $k\in\N_{\geq 2}$ such that $\partial^{\alpha} h(0)\neq 0$ for some $\alpha\in\N_0^{d-1}$ of length $|\alpha|\leq k$, see \cite[VII.3.2]{MR1232192}.
\end{remark}

\begin{proposition}\label{prop. schatten estimate low frequency}
Let $S\subset\R^d$ be a smooth hypersurface with normalized defining function $\rho:\R^d\to\R$.
Fix $\chi\in C^{\infty}(\R^d)$, and assume that 
\begin{align}\label{assp. decay of the FT}
\sup_{x\in\R^d}(1+|x|)^r|\widehat{\chi\rd\sigma_S}(x)|<\infty
\end{align} 
for some $r>0$, where $\rd\sigma_S$ is the canonical surface measure on $S$.
Let $1\leq q\leq 1+r$ and define 
\begin{align}\label{def. alpha_qrd}
\alpha_{q,r,d}:=\begin{cases}
\frac{2(d-1-r)q}{d-q},\quad &\mbox{if  }\frac{d}{d-r}\leq q\leq 1+r,\\
\frac{2rq+}{2rq-d(q-1)},\quad &\mbox{if  }1\leq q<\frac{d}{d-r}.
\end{cases}
\end{align}
Here, $2rq+$ means $2rq+\varepsilon$ with $\varepsilon>0$ arbitrarily small but fixed. Then for all $A,B\in L^{2q}(\R^2)$ and all $z\in\C\setminus\R$ we have the estimate
\begin{align}\label{inequality Schatten low frequency}
\|A\chi(D)(\rho(D)-z)^{-1}B\|_{\mathfrak{S}^{\alpha_{q,r,d}}}\leq C\|A\|_{2q}\|B\|_{2q},
\end{align}
with a constant $C$ independent of $z,A,B$.  
\end{proposition}


The proof relies on the following lemma. 

\begin{lemma}[pointwise bounds on complex powers]\label{lemma pointwise bounds on complex powers}
Let $h:\R^n\to \R$ be a smooth real-valued function and fix $\psi\in C^{\infty}(\R^n)$. Assume that
\begin{align}\label{assp. oscillatory integral}
\int_{\R^n}\e^{\I x h(\eta)}\psi(\eta)\rd\eta=\mathcal{O}((1+|x|)^{-r}),\quad x\in\R
\end{align}
for some $r>0$. Given $a\in [1,1+r]$, $t\in\R$, define the tempered distributions 
\begin{align}\label{def: uabpm}
u_{a,t}^{\pm}(\xi,\eta)=\e^{\pi^2(a+\I t)^2}\psi(\eta)(\xi-h(\eta)\pm \I 0)^{-a-\I t},\quad\xi\in \R,\quad\eta\in \R^n.
\end{align}
Then the $(n+1)$-dimensional inverse Fourier transform $v_{a,t}^{\pm}=\mathcal{F}^{-1}u_{a,t}^{\pm}$ satisfies the pointwise estimate
\begin{align*}
\sup_{t\in\R}\sup_{(x,y)\in \R\times\R^n}(1+|x|+|y|)^{1+r-a}|v_{a,t}^{\pm}(x,y)|<\infty.
\end{align*}
\end{lemma}

\begin{proof}
By a change of variables $\tau=\xi-h(\eta)$ and by \cite[Example 7.1.17]{MR1065993} the partial inverse Fourier transform of $(\xi-h(\eta)\pm \I 0)^{-a-\I t}$ with respect to the variable $\tau$ is given by 
\begin{align}\label{eq: partial inverse Fourier transform of complex power}
\mathcal{F}^{-1}_{\tau\to x}\{(\tau-h(\eta)\pm \I 0)^{-a-\I t}\}(x)=(2\pi)^{-1/2}\e^{\pm \I\pi(a+\I t)/2}\e^{\I x h(\eta)}\chi_{\mp}^{a+\I t-1}(x)
\end{align}
where $\chi_{\pm}^z$ are distributions on $\R$ given by
\begin{align*}
\chi_{\pm}^z(x)=\frac{x_{\pm}^z}{\Gamma(z+1)},\quad z\in\C.
\end{align*} 
Observing that $|\Gamma(a+\I t)|^{-1}\leq C\e^{\pi^2|a+\I t|^2/2}$, we get the desired bound  by combining \eqref{assp. oscillatory integral} and \eqref{eq: partial inverse Fourier transform of complex power}.
\end{proof}

\begin{proof}[Proof of Proposition \ref{prop. schatten estimate low frequency}]
The proof is similar to that of Lemma 4.4 in \cite{MR3608659}. The difference is that the pointwise estimamte (4.13) in \cite{MR3608659} is replaced by 
\begin{align}\label{eq. pointwise bound complex power of resolvent}
\sup_{x\in\R^d}(1+|x|)^{1+r-a}|\chi(D)(\rho(D)-z)^{-(a+\I t)}(x)|\leq C\e^{Ct^2},
\end{align}   
where $1\leq a\leq 1+r$ and $t\in\R$.
The above is a consequence of~\eqref{assp. decay of the FT} and Lemma~\ref{lemma pointwise bounds on complex powers} with $n=d-1$. This can be seen by invoking the implicit function theorem, by which we may locally write 
\begin{align}\label{factorization}
\rho(\xi)=e(\xi)(\xi_1-h(\xi')),\quad \xi=(\xi_1,\xi'),
\end{align}
in appropriate coordinates, where $e\neq 0$. Then $S$ is locally the graph of $h$ over the $\xi'$-plane and \eqref{assp. oscillatory integral} is satisfied.
Together with the Hardy-Littlewood-Sobolev inequality, \eqref{eq. pointwise bound complex power of resolvent} implies the Hilbert-Schmidt bound
\begin{align*}
\|A^{a+\I t}\chi(D)(\rho(D)-z)^{-(a+\I t)}B^{a+\I t}\|_{\mathfrak{S}^2}^2\leq 
C\e^{Ct^2}\|A\|_{\frac{2ad}{d-1-r+a}}\|B\|_{\frac{2ad}{d-1-r+a}}
\end{align*}
for $1\leq a\leq 1+r$. Complex interpolation with the trivial bound
\begin{align*}
\|A^{\I t}\chi(D)(\rho(D)-z)^{-\I t}B^{\I t}\|_{\mathfrak{S}^2}^2\leq 
C\e^{Ct^2}
\end{align*}
yields \eqref{inequality Schatten low frequency} for $\frac{d}{d-r}\leq q\leq 1+r$. The estimate for $1\leq q<\frac{d}{d-r}$ does not depend on the assumption \eqref{assp. decay of the FT}; the argument is similar to the case $1\leq q<\frac{2d}{d+1}$ in the proof of~\cite[Lemma 4.4]{MR3608659}. As in that case one proves
\begin{align*}
\|A\chi(D)(\rho(D)-z)^{-(b+\I t)}B\|_{\mathfrak{S}^1}\leq 
C\e^{Ct^2}(1-b)^{-1}\|A\|_2\|B\|_2,\quad 0<b<1
\end{align*}
and interpolates with
\begin{align*}
\|A\chi(D)(\rho(D)-z)^{-(1+r+\I t)}B\|_{\mathfrak{S}^2}\leq 
C\e^{Ct^2}\|A\|_2\|B\|_2,
\end{align*}
giving
\begin{align*}
\|A\chi(D)(\rho(D)-z)^{-1}B\|_{\mathfrak{S}^{1+}}\leq 
C\|A\|_2\|B\|_2.
\end{align*}
The last bound is again interpolated with the estimate \eqref{inequality Schatten low frequency} for $q$ in the range already proven. This yields \eqref{inequality Schatten low frequency} for the whole range of $q$; see Figure \ref{Figure}.
\end{proof}

\begin{remark}
In \cite[Lemma 4.4]{MR3608659} we have $\rho(\xi):=T(\xi)-\lambda$, and the estimate was proved under the assumption that $S$ has everywhere nonvanishing Gaussian curvature. The latter implies that \eqref{assp. decay of the FT} holds with $r=(d-1)/2$. The Schatten space estimates for the resolvent of the Laplacian were first proved by Frank-Sabin \cite{MR3730931}. By a duality argument, \eqref{inequality Schatten low frequency} implies 
\begin{align}\label{uniform Lp-Lp' by duality}
\|\chi(D)(\rho(D)-z)^{-1}\|_{L^p\to L^{p'}}\leq C,\quad \frac{1}{r+1}\leq \frac{1}{p}-\frac{1}{p'}\leq 1,
\end{align}
but is strictly stronger; in fact, \eqref{uniform Lp-Lp' by duality} is equivalent to \eqref{inequality Schatten low frequency} with $\alpha_{q,r,d}=\infty$. The estimate \eqref{uniform Lp-Lp' by duality} for the imaginary part of the resolvent follows already from a result of Greenleaf \cite[Theorem 1]{MR620265}; see also \cite[VII.5.15]{MR1232192}. 
\end{remark}

\begin{corollary}\label{cor. schatten estimate low frequency finite type}
Assume that $S\cap\supp\chi$ is compact and of finite type $k$. Then~\eqref{inequality Schatten low frequency} holds with $r=1/k$.
\end{corollary}

\begin{proof}
Under the finite type assumption \eqref{assp. decay of the FT} holds with $r=1/k$.
For $d=2$ this follows from Lemma \ref{lemma decay of the FT}. For $d\geq 3$ it follows from \cite[Theorem VIII.2]{MR1232192}. 
\end{proof}

\begin{corollary}\label{cor. schatten estimate low frequency principal curvature}
Assume that $S\cap \supp\chi$ is compact and has at least $l\leq d-1$ everywhere nonvanishing principal curvatures. Then \eqref{inequality Schatten low frequency} holds with $r=l/2$.
The constant is locally uniform in $\rho$ in the $C^2$-topology. 
\end{corollary}

\begin{proof}
Littman \cite{MR0155146} proved that the stated assumption implies the decay estimate~ \eqref{assp. decay of the FT} with $r=l/2$. The last claim follows from Lemma \ref{lemma stability curvature}.
\end{proof}

\begin{lemma}[Stability of curvature]\label{lemma stability curvature}
Let $S_1$ and $S_2$ be smooth hypersurfaces in $\R^d$ with normalized defining functions $\rho_1$ and~$\rho_2$. Assume that $S_1$ has $l\leq d-1$ nonvanishing principal curvatures at $\xi_0\in S_1$. Then there exists a neighborhood $U$ of $\xi_0$ in $\R^d$ and a constant $c>0$ such that whenever $\|\rho_1-\rho_2\|_{C^2(U)}\leq c$, then $S_2$ has $l$ nonvanishing principal curvatures in $S_2\cap U$.     
\end{lemma}

\begin{proof}
By assumption, we have
\begin{align}\label{stability curvature 1}
\sup_{\xi\in U}\|\nabla^2\rho_1-\nabla^2\rho_2\|\leq d^{1/2}c,
\end{align}
where $\|\cdot\|$ denotes the operator norm of $d\times d$-matrices. Let $U$ be a neighborhood of $\xi_0$ in $\R^d$ such that $\|\rho_1-\rho_2\|_{C^2(U)}\leq c$. We will later choose $\diam(U)$ and $c>0$ sufficiently small. Let $\xi_0\in S_1$ and $\eta_0\in S_2\cap U$. Using the triangle inequality, we get
\begin{align}\label{stability curvature 2}
|\nabla\rho_1(\xi_0)-\nabla\rho_2(\eta_0)|
\leq c+\|\rho_2\|_{C^2(U)}\diam(U).
\end{align}
Let $P_1$ and $P_2$ be the orthogonal projections in $\R^d$ onto the tangent spaces $T_{\xi_0}S_1=\{\nabla\rho_1(\xi_0)\}^{\perp}$ and $T_{\eta_0}S_2=\{\nabla\rho_2(\eta_0)\}^{\perp}$, respectively. Then by \eqref{stability curvature 2}, 
\begin{align}\label{stability curvature 3}
\|P_1-P_2\|^2=1-\langle \nabla\rho_1(\xi_0),\nabla\rho_2(\eta_0)\rangle_{\R^d}\leq c+\|\rho_2\|_{C^2(U)}\diam(U).
\end{align}
Set $A_1:=\nabla^2\rho_1(\xi_0)$ and $A_2:=\nabla^2\rho_2(\eta_0)$. Then \eqref{stability curvature 1} and \eqref{stability curvature 3} imply that
\begin{align*}
\|P_1A_1P_1-P_2A_2P_2\|&=\|P_2(A_1-A_2)P_2+P_2A_1(P_1-P_2)+(P_1-P_2)A_1P_1\|\\
&\leq \|A_1-A_2\|+2\|A_1\|\|P_1-P_2\|\\
&\leq d^{1/2}c+2\|\rho_1\|_{C^2(U)}(c+\|\rho_2\|_{C^2(U)}\diam(U))^{1/2}.
\end{align*}
By assumption, the spectrum of $P_1A_1P_1$ contains at least $l$ nonzero eigenvalues $\lambda_1,\ldots,\lambda_k$. Define $\varepsilon:=\min_{1\leq j\leq k}|\lambda_j|>0$, and choose $U$, $c$ such that
\begin{align*}
d^{1/2}c+2\|\rho_1\|_{C^2(U)}(c+\|\rho_2\|_{C^2(U)}\diam(U))^{1/2}\leq \varepsilon/2.
\end{align*}  
By the stability of bounded invertibility it follows that $P_2A_2P_2$ has at least $l$ nonzero eigenvalues. This proves the claim.
\end{proof}

The following is a version of \cite[Lemma 4.4]{MR3608659} near a nondegenerate critical point.

\begin{proposition}\label{prop. schatten estimate low frequency critical point}
Let $T:\R^d\to\R$ be a smooth function with a nondegenerate critical point $\xi_0$ and corresponding value $\lambda_{\rm c}\in\R$. Then there exists $\delta>0$ such that for fixed $\chi\in C_c^{\infty}(B_{\delta}(\xi_0))$, $d/2< q\leq (d+1)/2$ and for all $A,B\in L^{2q}(\R^2)$, $z\in\rho(T(D))$,  
we have the estimate
\begin{align}\label{inequality Schatten low frequency critical point}
\|A\chi(D)(T(D)-z)^{-1}B\|_{\mathfrak{S}^{\alpha_{q,r,d}}}\leq C|z-\lambda_{\rm c}|^{\frac{d}{2q}-1}\|A\|_{2q}\|B\|_{2q},
\end{align}
uniformly for $z$ in a punctured neighborhood of $\lambda_{\rm c}$. Here,
\begin{align*}
\alpha_{q,r,d}:=\begin{cases}
\frac{q(d-1)}{d-q},\quad &\mbox{if  }\frac{2d}{d+1}\leq q\leq \frac{d+1}{2},\\
\frac{q(d-1)+}{d-q},\quad &\mbox{if  }1\leq q<\frac{2d}{d+1}.
\end{cases}
\end{align*}
If $\xi_0$ is a local extremum, then the same estimate holds for $q=d/2$. If $\xi_0$ is a saddle point the same is true if $C$ is replaced by $-C\ln|z-\lambda_{\rm c}|$.
\end{proposition}

\begin{proof}
Without loss of generality we may assume that $\xi_0=0$ and $\lambda_{\rm c}=0$. We will pick $\delta$ at least so small that $B_{\delta}(0)$ does not contain any other critical point besides the origin. Let $Q$ be the quadratic form
\begin{align*}
Q(\xi):=\frac{1}{2}\langle\xi,H\xi\rangle_{\R^d}
\end{align*} 
where $H=\nabla^2T(0)$ is the Hessian of $T$ at the origin. By a linear change of variables we may assume that
\begin{align}\label{def. Q}
Q(\xi)=\xi_1^2+\ldots+\xi_j^2-\xi_{j+1}^2-\xi_d^2
\end{align}
where $2j-d$ is the signature of $Q$. By standard arguments involving the Phragmen-Lindel\"of maximum principle we may assume that $z=\lambda\pm\I 0$, where $\lambda\in\R$ lies in a small punctured neighborhood of the origin in $\R$. We will not need to distinguish between the two limits and so, by abuse of notation, we just write $\lambda$ to denote either of those limits.
Moreover, by possibly multiplying $T$ by $-1$ we may assume $\lambda>0$. By a change of scale $\xi=\lambda^{1/2}\eta$, it then suffices to prove that
\begin{align}\label{inequality Schatten low frequency critical point rescaled}
\|A\chi(D)(\lambda^{-1}T(\lambda^{1/2}D)-1)^{-1}B\|_{\mathfrak{S}^{\alpha_{q,r,d}}}\leq C\|A\|_{2q}\|B\|_{2q},
\end{align}
where $\chi$ is now supported in $B_{\delta\lambda^{-1/2}}(0)$. We have
\begin{align}\label{eq. Taylor T}
\lambda^{-1}T(\lambda^{1/2}\eta)-1=Q(\eta)-1+\lambda^{1/2}\mathcal{O}_{C^{n}}(|\eta|^3),\quad \eta\in B_{\delta\lambda^{-1/2}}(0)
\end{align}
for any fixed $n$, which we assume to be sufficiently large. Here, $g=\mathcal{O}_{C^{n}}(|\xi|^3)$ means $\partial^{\alpha}g=\mathcal{O}(|\xi|^{(3-|\alpha|)_+})$ for $|\alpha|\leq n$ as $\xi\to 0$. If we set $M_{\lambda}:=\set{\xi\in\R^d}{T(\xi)=\lambda}$, then $\rho(\eta):=\lambda^{-1}T(\lambda^{1/2}\eta)-1$ is an approximately normalized defining function of  $S:=\lambda^{-1/2}M_{\lambda}$ in the sense that $c\leq |\nabla\rho|\leq 1/c$ on $S$ with a constant $c>0$ that is uniform for $\lambda>0$ in a small neighborhood of the origin; hence we may as well assume $\rho$ to be normalized. 

We first consider the cases when the signature of $Q$ is $d$ or $-d$. Since these two cases differ from each other only by a change of sign we only treat the first one. Since $|\rho(\eta)|\geq 1/2$ for $|\eta|\leq 1/2$ or $|\eta|\geq 3/2$ and for $\delta$ sufficiently small, we may restrict our attention to the region $A:=\set{\eta\in\R^d}{1/2\leq |\eta|\leq 3/2}$, i.e.\ we assume that $\chi$ localizes to this region. In view of \eqref{eq. Taylor T}, $\rho$ is an $\mathcal{O}_{C^n(A)}(\lambda^{1/2})$-perturbation of $\rho_1(\eta):=Q(\eta)-1$. Since $\rho_1$ is an approximately normalized defining function of the unit sphere we can apply Corollary \ref{cor. schatten estimate low frequency principal curvature} to conclude \eqref{inequality Schatten low frequency critical point rescaled}.   

Now consider the case when the signature of $Q$ is different from $\pm d$, i.e.\ when $j\in \{1,\ldots,d-1\}$ in \eqref{def. Q}. The set $\set{\eta\in\R^d}{Q(\eta)=1}$ is a two-sheeted hyperboloid, which in contrast to the sphere in the previous case is noncompact. While the localization $\chi$ introduces a cutoff to the frequency scale $|\eta|\leq\delta\lambda^{-1/2}$, the latter is not uniform with respect to small $\lambda$, so we cannot apply Corollary \ref{cor. schatten estimate low frequency principal curvature}.
Instead, we will prove that a modification of \eqref{assp. oscillatory integral} holds and appeal to (the proof of) Proposition \ref{prop. schatten estimate low frequency} directly. By the implicit function theorem we may solve the equation $\rho(\eta)=0$ for one of the first $j$ variables. By a partition of unity of the $j-1$ sphere we may thus assume that $\rho(\eta)=e(\eta)(\eta_1-h(\eta'))$ on the set where $|\rho(\eta)|\leq 1/2$. Let $1=\sum_{k=0}^{-c\ln\lambda}\phi_k(\eta')$ be a Littlewood-Paley decomposition of the set $|\eta'|\leq \delta\lambda^{-1/2}$. We claim that
\begin{align}\label{pointwise estimate at frequency k II}
|\mathcal{F}^{-1}\{\phi_k(\rho\pm \I 0)^{-a-\I t}\}(x)|\lesssim 2^{k(d-2a)}(1+2^k|x|))^{-\frac{d-1}{2}}|2^k x|^{a-1}.
\end{align}
By a change of variables and in view of the obvious lower bound $|e(\eta)|\gtrsim |\eta|$ the proof would follow from the pointwise estimate
\begin{align}\label{pointwise estimate at frequency k}
\int_{\R^{d-1}}\e^{\I x h(2^k\eta')}\psi(\eta')\rd\eta'=\mathcal{O}((1+2^k|x|))^{-\frac{d-1}{2}},\quad x\in\R
\end{align}
analogously to the proof of Lemma \ref{lemma pointwise bounds on complex powers}.
Summing \eqref{pointwise estimate at frequency k II} over $0\leq k\leq -c\ln \lambda$ we get
\begin{align*}
\mathcal{F}^{-1}\{\chi(\rho\pm \I 0)^{-a-\I t}\}\in L^{\infty}(\R^d),\quad \frac{d}{2}<a\leq \frac{d+1}{2},
\end{align*}
locally uniformly in $\lambda$. At the endpoint $a=d/2$ one obtains an $\mathcal{O}(-\ln\lambda)$ bound due to logarithmic divergence of the sum. The proof of \eqref{pointwise estimate at frequency k} follows by standard stationary phase arguments upon observing that the Hessian of $2^{-k}h(2^k\eta')$ is uniformly nondegenerate; this may be seen by differentiating the implicit equation $\rho(h(\eta'),\eta')=0$ twice.
\end{proof}

\bibliographystyle{abbrv}
\bibliography{C:/Users/Jean-Claude/Dropbox/papers/bibliography_masterfile}
\end{document}